\documentclass[10pt,a4paper]{article}

\usepackage{amsmath, amssymb, latexsym}
\usepackage{amsopn,amsfonts, amsthm}
\usepackage[T1]{fontenc}

\usepackage[english]{babel}
\usepackage{blindtext}
\usepackage{courier}
\usepackage[scaled=.90]{helvet}
\usepackage{xcolor}
\usepackage{titlesec}

\usepackage{amsthm,euscript,epsf}
\usepackage{float}
\usepackage{fancyhdr}
\usepackage{setspace}
\usepackage{multirow}
\usepackage[Lenny]{fncychap}
\usepackage[square,sort&compress,numbers]{natbib}

\newtheorem{theorem}{Theorem}

\newtheorem{corollary}[theorem]{\small Corollary}

\theoremstyle{definition}

\newtheorem{definition}[theorem]{\small Definition}
\newtheorem{example}[theorem]{\small Example}

\title{$C$-open Sets on Topological Spaces }
\author{Mesfer H. Alqahtani\footnote{University College of Umluj, Mathematics Department, University of Tabuk, Tabuk, Saudi Arabia, E-mail:m\b{ }halqahtani@ut.edu.sa }}

\begin{document}

\maketitle
\begin{abstract}  An open (resp., closed) subset $A$ of a topological space $(X, \mathcal{T})$ is called {\it $C$-open} (resp., $C$-closed) set if $ cl(A)\setminus A $ (resp., $ A\setminus int(A) $) is a countable set. This paper aims to present the concept of $C$-open and $C$-closed sets. We first investigate their basic properties. Then, we found some operators such as interior, closure, limit, border, and frontier using $C$-open and $C$-closed sets.
The relationships between them are clarified and discussed. Finally, we  exhibit continuous maps and compact space defined using $C$-open and $C$-closed sets and scrutinize their main properties.\\
 

\end{abstract}
\section{Introduction}In the topological space $X$ a subset $B$ of a space $X$ is said to be a {\it regularly-open}, called also {\it open domain} if it is the interior of its own closure \cite{Kuratowski}. A subset $B$ is said to be a {\it regularly-closed}, called also {\it closed domain} if it is the closure of its own interior or if its complement is an open domain. An open (resp., closed) subset $A$ of a topological space $(X, \mathcal{T})$ is called {\it $F$-open} (resp., $F$-closed) set if $ cl(A)\setminus A $ (resp., $ A\setminus int(A) $) is finite set \cite{Mesfer2}. In this present paper, we have concerned is to introduce and study the notions of $C$-open and $C$-closed sets in topological spaces that are independent of regularly-open and regularly-closed sets. In Section 2, we have defined $C$-open and $C$-closed sets and investigated their basic properties. Then we have shown about the relationship between them with the other kinds such as open domain, closed domain, $F$-open, $F$-closed, closed and open sets in topological spaces. In Section 3, we have established some operators such as $C$-interior, $C$-closure, $C$-limit, $C$-border and $C-$frontier using $C$-open and $C$-closed sets and the relationships between them are clarified and discussed. In Section 4, we have displayed $C$-continuous functions and $C$-compact sets defined using $C$-open and $C$-closed sets and examine their basic properties. In addition, I have shown the $C$-continuious, $C$-open and onto image of $C$-compact space is $C$-compact space. Throughout this paper, the subset $B$ of a space $X$, I will denote the complement of $B$ in $X$ by $X\setminus B$, the set of positive integers numbers by $\mathbb{N}$, the set of integers numbers by $\mathbb{Z}$, the set of rational numbers by $\mathbb{Q}$, the set of irrational numbers by $\mathbb{I}$, the set of real numbers by $\mathbb{R},$ and usual topology in $\mathbb{R}$ by $\mathcal{U}$ \cite{Steen}. Unless or otherwise mentioned, $X$ stands for the topological space $(X, \mathcal{T}).$ We do not assume $T_{2}$ in the definition of compactness and countable compactness. We also do not assume regularity in the definition of $F-$Lindel$\ddot{o}$f.

\section{Main properties of $C$-open sets}

\begin{definition}\label{13p8}

An open subset $A$ of a topological space $X$ is called {\it $C$-open} set if $ cl(A)\setminus A $ is countable set. That is, $A$ is an open set and the frontier of $A$ is a countable set.
\end{definition}

\begin{definition}\label{13p9}
A closed subset $A$ of a topological space $(X,\mathcal{T})$ is called {\it $C$-closed} set if $A\setminus int(A)$ is countable set. That is, $A$ is a closed set and the frontier of $A$ is countable set.
\end{definition}
We denoted for the collection of all $C$-open (resp., $C$-closed) subsets of a topological space $(X,\mathcal{T})$ by $CO(X)$ (resp., $CC(X)$).
\begin{theorem}\label{13p4}
  The complement of any $C$-open (resp., $C$-closed) subset of a topological space $(X,\mathcal{T})$ is a $C$-closed (resp., $C$-open) set.
\end{theorem}
\begin{proof}
 Let $A$ be a any $C$-open subset of a topological space $(X,\mathcal{T}).$ Then $X\setminus A$ is closed and $(X\setminus A)\setminus int(X\setminus A)=(X\setminus A)\setminus (X\setminus cl(A))=(X\setminus A)\cap cl(A)=cl(A) \setminus A$ is countable, because $A$ is $C$-open. Therefore, $X\setminus A$ is $C$-closed set. On the other hand, suppose that $A$ be a any $C$-closed subset of a topological space $(X,\mathcal{T}).$ Then $X\setminus A$ is open and $cl(X\setminus A) \setminus( X\setminus A)=(X\setminus X\setminus A)\cap cl(X\setminus A)=A\cap cl(X\setminus A)=A\setminus (X\setminus cl(X\setminus A))=A\setminus int(A)$ is countable, because $A$ is $C$-closed. Therefore, $X\setminus A$ is $C$-open set.
  
\end{proof}

 Obviously, from the definitions that any $F$-open (resp., $F$-closed) subset of a topological space $(X,\mathcal{T})$ is $C$-open (resp., $C$-closed) set. However, the converse always is not true. For example, the subset $\mathbb{R}\setminus \mathbb{Z}$ is a $C$-open subset of the usual topological space $(\mathbb{R},\mathcal{U})$, which is not $F$-open. Also by Theorem \ref{13p4}, any complement of $C$-open set is $C$-closed, then $\mathbb{Z}$ is a $C$-closed subset of the usual topological space  $(\mathbb{R},\mathcal{U})$, which is not $F$-closed, because $\mathbb{Z}\setminus int(\mathbb{Z})=\mathbb{Z}\setminus \emptyset=\mathbb{Z}$ is not finite. It is clear by the definitions that, any clopen (closed-and-open) subset of a topological space $(X,\mathcal{T})$ is $C$-clopen set. Also any countable closed set is $C$-closed. However, any countable open set may not be $C$-open. For example, $\mathbb{R}$ with particular point topology at $1.$ We have, $\mathbb{N}$ is a countable open set, but $cl(\mathbb{N})\setminus \mathbb{N}=\mathbb{R}\setminus \mathbb{N}$ is uncountable set.\\
     
In $(\mathbb{R},\mathcal{U})$ any open intervals is $C$-open set. Also any closed intervals is $C$-closed set. It is clear by the definitions \ref{13p8} and \ref{13p9}, every $C$-open and $C$-closed sets are open and closed sets, respectively. However, the converse always is not true. Here is an example of open (resp., closed) set which is not $C$-open (resp., $C$-closed).
\begin{example}\label{14p1}
 Let $(\mathbb{R},\mathcal{T_{\mathbb{I}}})$ be the excluded set topological space on $\mathbb{R}$ by $\mathbb{I}.$ Then $\mathbb{Q}$ is open in $(\mathbb{R},\mathcal{T_{\mathbb{I}}})$. But $cl(\mathbb{Q})\setminus \mathbb{Q}=\mathbb{R}\setminus \mathbb{Q}=\mathbb{I}$ is uncountable set. Hence, $\mathbb{Q}$ is not $C$-open set. Also, $\mathbb{I}$ is an example of closed set which is not $C$-closed set.
\end{example}

There is an example of $C$-open (resp., $C$-closed) set is not an open (resp., closed) domain set.
\begin{example}
 Let $A=(2,5)\cup (5,9)$ is a $C$-open subset in $(\mathbb{R},\mathcal{U})$. However, $A$ is not open domain.
Moreover, $\mathbb{R}\setminus A$ is $C$-closed set, but $\mathbb{R}\setminus A$ is not closed domain.
\end{example}

 We provide an example of an open (resp., closed) domain set is not $C$-open (resp., $C$-closed) set.

 \begin{example}\label{13p3}
By Example \ref{14p1}, let $H=\{1,2,3\}$, then $int(cl(H))=int(cl(\{1,2,3\}))=int(\{1,2,3\}\cup \mathbb{I})=\{1,2,3\}=H$ is open domain. But $cl(H) \setminus H=cl(\{1,2,3\})\setminus \{1,2,3\}=(\{1,2,3\}\cup \mathbb{I})\setminus \{1,2,3\}=\mathbb{I}$ is uncountable set, then $H$ is not $C$-open set. Moreover, let $K=\mathbb{R}\setminus \{1,2,3\} ,$ then $cl(int(K))=cl(int(\mathbb{R}\setminus \{1,2,3\}))=cl(\mathbb{Q}\setminus\{1,2,3\})=\mathbb{I} \cup (\mathbb{Q}\setminus\{1,2,3\})=\mathbb{R}\setminus \{1,2,3\}=K$ is closed domain. But $K \setminus int(K)=(\mathbb{R}\setminus \{1,2,3\})\setminus int(\mathbb{R}\setminus \{1,2,3\})= (\mathbb{R}\setminus\{1,2,3\})\setminus ( \mathbb{Q}\setminus\{1,2,3\})=\mathbb{I}$ is uncountable set, then $K$ is not $C$-closed set.
 \end{example}

  For a topological space $(X,\mathbb{T}),$ we always have:
  \begin{center}
     $F$-open set $\Longrightarrow$ $C$-open set $\Longrightarrow$ open set \\
      $F$-closed set $\Longrightarrow$ $C$-closed set $\Longrightarrow$ closed set
  \end{center}
 None of the above implications is reversible.

\begin{theorem}\label{aaaa}
The finite unions of $C$-open sets is $C$-open.

\end{theorem}
\begin{proof}
Suppose that $K_{i}$ be a $C$-open set for all $i \in \{1,2,3,\dots,n\}$. Then $K_{i}$ is an open set and $cl(K_{i})\setminus K_{i}$ is countable for all $i$. Since $\bigcup_{i=1}^{n}K_{i}$ is open, then we need to show the other condition of $C$-open set. Now, we have, $$cl(\bigcup_{i=1}^{n} K_{i})\setminus\bigcup_{i=1}^{n} K_{i}= \bigcup_{i=1}^{n}(cl(K_{i})\setminus K_{i}) ,$$
       Since the finite union of countable sets is countable. Then, $cl(\bigcup_{i=1}^{n}K_{i}) \setminus \bigcup_{i=1}^{n}K_{i}$ is countable. Therefore, $\bigcup_{i=1}^{n}K_{i}$ is $C$-open.

\end{proof}

By Theorem \ref{aaaa}, Theorem \ref{13p4} and by Morgan's Laws, we have the following Corollary:
\begin{corollary}
  The finite intersections of $C$-closed sets is $C$-closed.
\end{corollary}

\section{Some operators via $C$-open and $C$-closed sets}

We recall some symbol in \cite{Mesfer2}. Let $K$ be any subset of the topological space $(X ,\mathcal{T}),$ and via the concepts of $F$-open sets we denoted for the interior by ${\rm int^{F}}(K)$, the closure by ${\rm cl^{F}}(K)$, the border by ${\rm Bd^{F}}(K)$, the frontier by ${\rm Fr^{F}}(K)$, the exterior by ${\rm Ext^{F}}(K)$ and the derived set by ${\rm D^{F}}(K)$. 
\begin{definition} \label{13p7}
 Let $K$ be a subset of the topological space $(X \mathcal{T}).$ Then,
  \begin{enumerate}
    \item [i)] The {\it $C$-interior} of $K$ is defined as the union of all $C$-open subsets of $K$, and is denoted by $int^{C}(K)$ or the largest $C$-open set contained in $K.$
    \item [ii)] The {\it $C$-clouser} of $K$ is defined as the intersection  of all $C$-closed sets containing $K$, and is denoted by $cl^{C}(K)$.
  \end{enumerate}

\end{definition}
\begin{definition}
  Let $x \in int^{C}(K)$ if and only if there exists a $C$-open set containing $x$ and contained in $K$.

\end{definition}
\begin{theorem}\label{13p14}
   Let $K$ be a subset of the topological space $(X,\mathcal{T}).$ Then,

  \begin{enumerate}
    \item [i)] $int^{C}(K)\subseteq int(K)\subseteq K.$
    \item [ii)] $K \subseteq cl(K)\subseteq cl^{C}(K).$
  \end{enumerate}

\end{theorem}
\begin{proof}
   \begin{enumerate}
    \item [i)] Let $U$ be the largest $C$-open set contained in $K$, i.e. $int^{C}(K)=U$. Since any $C$-open set is open, then $U \subseteq int(K)$. Hence $int^{C}(K)\subseteq int(K)$. By the definition of interior of $K$, we have $int(K)\subseteq K.$ Therefore, $int^{C}(K)\subseteq int(K)\subseteq K.$
    \item [ii)] By the definition of closure, we know that $K \subseteq cl(K)$. Now, let $cl(K)=F$ is the smallest closed set containing $K$, then $F \subseteq cl^{C}(K)$. Suppose not, $F \nsubseteq cl^{C}(K)$, then there exist $C$-closed set $E$ such that $cl^{C}(K)=E$ and $K\subseteq E$. But $E$ is closed, because any $C$-closed is closed. Thus contradiction. Hence, $K \subseteq cl(K)\subseteq cl^{C}(K).$

  \end{enumerate}
\end{proof}
In general $int(K)\nsubseteq int^{C}(K)$ and $cl^{C}(K)\nsubseteq cl(K)$. For example:
\begin{example}
 By Example \ref{14p1} we have the open set $\mathbb{Q}$ is neither $F$-open nor $C$-open set. Then $int(\mathbb{Q})=\mathbb{Q}$, but $int^{C}(\mathbb{Q})\subsetneqq \mathbb{Q}$. Therefore, $int(K)\nsubseteq int^{C}(K).$ By Example \ref{14p1} we have $\mathbb{I}$ is not $C$-closed set and $cl^{C}(\mathbb{I})= \mathbb{R}$, but $cl(\mathbb{I})=\mathbb{I}.$ Therefore, $ cl^{C}(K)\nsubseteq cl(K)$.
\end{example}
\begin{definition}\label{13p13}
 Let $K$ be a subset of the topological space $(X,\mathcal{T})$. A point $x\in X$ is said to be {\it $C$-limit points} (or an {\it $C$-accumulation point}, or a {\it $C$-cluster point}) of $K$ if and
only if every $C$-open set $V$ containing $x,$ contains at least one point of $K$
different from $x$. The set of all $C$-limit points of $K$ is called the {\it $C-$derived set} of $K$ and denoted by $D^{C}(K).$

\end{definition}
\begin{theorem}
 Let $K$ and $H$ be subsets of a topological space $(X \mathcal{T}).$ Then we have the following topological properties:
  \begin{enumerate}
    \item [(i)] $D(K)\subset D^{C}(K)$, where $D(K)$ is the derived set of $K.$
    \item [(ii)] If $K\subseteq H,$ then $D^{C}(K)\subseteq D^{C}(H).$
     \item [(iii)] $D^{C}(K)\cup D^{C}(H)=D^{C}(K\cup H)$ and $D^{C}(K\cap H)\subset  D^{C}(K)\cap D^{C}( H).$
  \end{enumerate}
\end{theorem}
\begin{proof}
 For $(i)$ it suffices to observe that every $C$-open is open. For $ (ii)$ follow from Definition \ref{13p13}. For $(iii)$ is a modification of the standard proof for $D$, where open sets are replaced by $C$-open sets.
\end{proof}
In general $D^{C}(K) \nsubseteq D(K)$. For example:
\begin{example}
  Let $(\mathbb{R},\mathcal{T}_{\frac{1}{2}})$ be a topological space, where $\mathcal{T}_{\frac{1}{2}}$ is the particular point topology at $\frac{1}{2}$. Let $K=\mathbb{R}\setminus\{\frac{1}{2}\}$, then $D(\mathbb{R}\setminus\{\frac{1}{2}\})=\emptyset$, because $\{\frac{1}{2},x\}$ is an open set for any $x \in \mathbb{R} $ and $(\{\frac{1}{2},x\}\setminus \{x\})\cap (\mathbb{R}\setminus\{\frac{1}{2}\})=\emptyset.$ Since any nonempty $C$-open set in $(\mathbb{R},\mathcal{T}_{\frac{1}{2}})$ is of the form $\mathbb{R}\setminus H$, where $H$ is any countable set and $\frac{1}{2} \notin H.$ let $x \in \mathbb{R}$ be arbitrary, then any $C$-open set containing $x$ contains an other point of $\mathbb{R}\setminus\{\frac{1}{2}\}$
different from $x$. Hence $D^{C}(\mathbb{R}\setminus\{\frac{1}{2}\})=\mathbb{R}.$ Then $D^{C}(\mathbb{R}\setminus\{\frac{1}{2}\}) \nsubseteq D(\mathbb{R}\setminus\{\frac{1}{2}\}).$ Therefore, $D^{C}(K) \nsubseteq D(K).$
\end{example}
In general $D^{C}(K\cap H)\nsupseteq  D^{C}(K)\cap D^{C}( H).$ For example:
\begin{example}
  Let $K=(1,2)$ and $H=(2,3)$ be a subsets in $(\mathbb{R},\mathcal{U})$. Since any $C$-open set in $(\mathbb{R},\mathcal{U})$ is an open interval or $\mathbb{R}\setminus F$, where $F$ is any countable closed set, then $D^{C}(K\cap H)=D^{C}((1,2)\cap (2,3))=D^{C}(\emptyset)=\emptyset$ and $D^{C}(K)\cap D^{C}( H)=D^{C}((1,2))\cap D^{C}((2,3))=[1,2]\cap [2,3]=\{2\}$. Then  $D^{C}(K\cap H)\nsupseteq  D^{C}(K)\cap D^{C}( H).$

\end{example}

\begin{theorem}\label{13p16}
 Let $K$ and $H$ be subsets of a topological space $(X \mathcal{T}).$ Then we have the following topological properties:
\begin{enumerate}
  \item [(i)] $int^{C}(X)=X.$
  \item [(ii)] $int^{C}(K)\subseteq K.$
  \item [(iii)] If $K\subseteq H$, then $int^{C}(K)\subseteq int^{C}(H)$.
  \item [(iv)] $int^{C}(int^{C}(K))=int^{C}(K).$
  \item [(v)] $int^{C}(K \cap H)=int^{C}(K) \cap int^{C}(H).$
  \item [(vi)] $int^{C}(K) \cup int^{C}(H)\subseteq int^{C}(K \cup H).$
\end{enumerate}

\end{theorem}
\begin{proof}
  The properties $(i),(ii), (iii)$ and $ (iv)$ follow from Definitions \ref{13p8} and Definition \ref{13p7}. To prove $(v)$, by property $(ii)$ we have $int^{C}(K )\subseteq K$ and $int^{C}(H )\subseteq H,$ then $int^{C}(K )\cap int^{C}(H )\subseteq K \cap H. $ As $int^{C}(K )\cap int^{C}(H )$ is $C-$open, then we have $int^{C}(K )\cap int^{C}(H )\subseteq int^{C}(K \cap H), $ because $int^{C}(K )\cap int^{C}(H )$ is $C$-open and $int^{C}(K \cap H)$ is the largest $C$-open set contained in $K \cap H.$ Conversely, $(K \cap H)\subseteq K$ and $(K \cap H)\subseteq H$, by property $(iii)$ we have $int^{C}(K \cap H)\subseteq int^{C}(K)$ and $int^{C}(K \cap H)\subseteq int^{C}(H)$, hence $int^{C}(K \cap H)\subseteq int^{C}(K)\cap int^{C}(H)$. Therefore, $int^{C}(K \cap H)=int^{C}(K) \cap int^{C}(H).$ To prove $(vi)$, since $K \subseteq (K \cup H)$ and $H \subseteq (K \cup H)$, from property $(iii)$ we have $int^{C}(K) \subseteq int^{C}(K \cup H)$ and $int^{C}(H) \subseteq int^{C}(K \cup H). $ Therefore, $int^{C}(K) \cup int^{C}(H)\subseteq int^{C}(K \cup H).$
\end{proof}
In general, $int^{C}(K \cup H)\nsubseteq int^{C}(K) \cup int^{C}(H). $ For example:
\begin{example}

  Let $K=\mathbb{Q}$ and $H=\mathbb{I}$ be a subsets in $(\mathbb{R},\mathcal{U})$, then $int^{C}(K \cup H)=\mathbb{R}$. But $\mathbb{Q}$ and $ \mathbb{I}$  are not $C$-open sets, then $int^{C}(\mathbb{Q})\subset \mathbb{Q}$ and $int^{C}(\mathbb{I})\subset \mathbb{I}$. Therefore, $int^{C}(K \cup H)\nsubseteq int^{C}(K) \cup int^{C}(H). $

\end{example}
\begin{theorem}
 Let $K$ and $H$ be subsets of a topological space $(X, \mathcal{T}).$ Then we have the following topological properties:

\begin{enumerate}
  \item [(i)] $cl^{C}(\emptyset)=\emptyset.$
  \item [(ii)] $K\subseteq cl^{C}(K).$
   \item [(iii)] If $K\subseteq H$, then $cl^{C}(K)\subseteq cl^{C}(H)$.
  \item [(iv)] $cl^{C}(K \cup H)=cl^{C}(K) \cup cl^{C}(H).$
  \item [(v)] $cl^{C}(cl^{C}(K))=cl^{C}(K).$
\end{enumerate}

\end{theorem}

\begin{proof}
  The properties $(i),(ii), (iii)$ and $ (v)$ follow from Definition \ref{13p9} and Definition \ref{13p7}. The property $(iv)$, follow from property $(iii)$, Definition \ref{13p8}, Definition \ref{13p7} and using set theoretic properties.
\end{proof}
\begin{theorem}\label{13p15}
 Let $K$ be a subset of a topological space $(X \mathcal{T}).$ Then we have the following topological properties:
\begin{enumerate}
  \item [(i)] $int^{C}(K)= X\setminus cl^{C}(X \setminus K).$
  \item [(ii)] $ cl^{C}(K)= X\setminus int^{C}(X \setminus K).$
\end{enumerate}

\end{theorem}
\begin{proof}
  \begin{enumerate}
  \item [(i)] We have $int^{C}(K)= X\setminus cl^{C}(X \setminus K)$, then $X \setminus K \subseteq cl^{C}(X \setminus K)$, thus $X \setminus cl^{C}(X \setminus K) \subseteq K  .$ Since $X \setminus cl^{C}(X \setminus K)$ is $C$-open, then $X \setminus cl^{C}(X \setminus K)\subseteq int^{C}(K)$...(1). Now, let $V$ be any $C$-open set contained in $K$, i.e., $V\subseteq K$ and $V$ is $C$-open, then $X \setminus K \subseteq X \setminus V=cl^{C}(X \setminus V)$. Then $cl^{C}(X \setminus K) \subseteq X \setminus V$, hence $V\subseteq X \setminus cl^{C}(X \setminus K)$.
That is, any $C$-open set contained in $K$ is contained in $ X \setminus cl^{C}(X \setminus K)$, which means
that $int^{C}(K)\subseteq X\setminus cl^{C}(X \setminus K)$ ...(2). From (1) and (2) equality holds.

  \item [(ii)] Can be proved by replacing $K$ and $X \setminus K$ by $X \setminus K$ and $K$, respectively in $(i)$ and using set theoretic properties.
\end{enumerate}
\end{proof}

\begin{definition}\label{13p10}
 Let $K$ be a subset of a topological space $(X \mathcal{T}).$ Then {\it $C$-border} of $K$ is defined as $Bd^{C}(K)=K\setminus int^{C}(K).$

\end{definition}
\begin{theorem}
  Let $K$ be a subset of a topological space $(X \mathcal{T}).$ Then we have the following properties:
  \begin{enumerate}
    \item [(i)] $Bd(K)\subset Bd^{C}(K)$, where $Bd(K)$ denotes the border of $K.$
    \item [(ii)] $K= int^{C}(K) \cup Bd^{C}(K).$
    \item [(iii)]  $int^{C}(K) \cap Bd^{C}(K)=\emptyset.$
    \item [(iv)] $K$ is a $C$-open set if and only if $Bd^{C}(K)=\emptyset.$
    \item [(v)] $int^{C}(Bd^{C}(K))=\emptyset.$
    \item [(vi)] $Bd^{C}(Bd^{C}(K))=Bd^{C}(K).$
     \item [(vii)] $Bd^{C}(K)=K \cap cl^{C}(X\setminus K).$

  \end{enumerate}
\end{theorem}
\begin{proof}
To prove $(i)$, let $x \in Bd(K) $ be arbitrary, then $x \in K\setminus int(K)$, then $x \notin int(K)$. By Theorem \ref{13p14} part $(i)$ we have $int^{C}(K)\subseteq int(K) \subseteq K$, then $x\notin int^{C}(K)$. Hence $x \in (K \setminus int^{C}(K))=Bd^{C}(K)$. Therefore, $Bd(K)\subseteq Bd^{C}(K)$. The properties $(ii), (iii)$ and $ (iv)$ follow from Definition \ref{13p10}. To prove $(v)$, let $x \in int^{C}(Bd^{C}(K))$ be arbitrary, then $x \in Bd^{C}(K).$ Since $Bd^{C}(K)\subseteq K,$ then $x \in int^{C}(Bd^{C}(K))\subseteq int^{C}(K),$ then $x \in int^{C}(K) \cap Bd^{C}(K)$ which contradicts $(iii).$ For $(vi)$ from Definition \ref{13p10} and property $(v)$ we have $Bd^{C}(Bd^{C}(K))= Bd^{C}(K)\setminus int^{C}(Bd^{C}(K))=Bd^{C}(K)\setminus \emptyset =Bd^{C}(K)$. The property $(vii)$ follow from Theorem \ref{13p15} part $(i)$ and Definition \ref{13p10}.
\end{proof}

In general $Bd^{C}(K)\nsubseteq Bd(K).$ For example:
\begin{example}By Example \ref{14p1}, there exists $K=\mathbb{Q}$ is an open set which is not $C$-open set.
 Then $Bd(K)=\emptyset$. Since $\mathbb{Q}$  is not $C$-open set, then $int^{C}(\mathbb{Q}) \subsetneqq \mathbb{Q}.$ Hence, $Bd^{C}(K)=\mathbb{Q}\setminus int^{C}(\mathbb{Q})\neq \emptyset.$ Therefore, $Bd^{C}(K)\nsubseteq Bd(K).$

\end{example}

\begin{definition}\label{13p11}
 Let $K$ be a subset of a topological space $(X \mathcal{T}).$ Then {\it $C$-frontier} of $K$ is defined as $Fr^{C}(K)=cl^{C}(K)\setminus int^{C}(K).$

\end{definition}
\begin{theorem}
  Let $K$ be a subset of a topological space $(X \mathcal{T}).$ Then we have the following properties:
  \begin{enumerate}
    \item [(i)] $Fr(K)\subset Fr^{C}(K)$, where $Fr(K)$ denotes the frontier of $K.$
    \item [(ii)] $cl^{C}(K)= int^{C}(K) \cup Fr^{C}(K).$
    \item [(iii)]  $int^{C}(K) \cap Fr^{C}(K)=\emptyset.$
    \item [(iv)] $Bd^{C}(K)\subset Fr^{C}(K).$
    \item [(v)] $Fr^{C}(K)=cl^{C}(K) \cap cl^{C}(X\setminus K).$
    \item [(vi)]   $Fr^{C}(K)=Fr^{C}(X\setminus K).$
 \item [(viii)] $int^{C}(K)=K\setminus Fr^{C}(K).$
  \end{enumerate}
\end{theorem}

\begin{proof}
  The property $(i)$, by Theorem \ref{13p14} and Definition \ref{13p11}, we have $int^{C}(K)\subseteq int(K)$ and $cl(K)\subseteq cl^{C}(K),$ then $Fr(C)=(cl(K)\setminus int(K))\subseteq cl^{C}(K)\setminus int^{C}(K)=Fr^{C}(K).$ The properties $(ii)$ and $ (iii)$ follow from Definition \ref{13p11}. For $(iv)$, since $K\subseteq cl^{C}(K)$, then $Bd^{C}(K)= (K\setminus int^{C}(K))\subseteq (cl^{C}(K)\setminus int^{C}(K))=Fr^{C}(K).$ The property $(v)$ follow from Theorem \ref{13p15} and Definition \ref{13p11}. The property $(vi)$ follow from property $(v)$ and Definition \ref{13p11}. The property $(vii)$ follow from Definition \ref{13p11}.

\end{proof}

In general $Fr^{C}{K}\nsubseteq Fr(K)$ and $Fr^{C}(K)\nsubseteq Bd^{C}(K).$ For example:
\begin{example}
By Example \ref{14p1}, there exists $K= \mathbb{Q}$ such that $Fr(\mathbb{Q})=\mathbb{I}$. Since $\mathbb{Q}$ is not $C$-open set, then $int^{C}(\mathbb{Q})\subsetneqq \mathbb{Q}$, hence  $Fr^{C}{\mathbb{Q}}=(cl^{C}(\mathbb{Q})\setminus int^{C}(\mathbb{Q}))=\mathbb{R}\setminus int^{C}(\mathbb{Q})=\mathbb{I} \cup H$, where $H$ is nonempty subset of $\mathbb{Q}.$ Therefore, $Fr^{C}(K)\nsubseteq Fr(K).$ For $Fr^{C}(K)\nsubseteq Bd^{C}(K)$, let $K= (1,2]$ be a subsets in $(\mathbb{R},\mathcal{U})$, then $Bd^{C}(K)=\{2\} $ and $Fr^{C}(K)=\{1,2\} $. Hence, $Fr^{C}(K)\nsubseteq Bd^{C}(K).$

\end{example}

\begin{definition}\label{13p12}
 Let $K$ be a subset of a topological space $(X \mathcal{T}).$ Then {\it $C$-exterior} of $K$ is defined as $Ext^{C}(K)=int^{C}(X\setminus K).$

\end{definition}
\begin{theorem}
   Let $K$ be a subset of a topological space $(X \mathcal{T}).$ Then we have following properties:
  \begin{enumerate}
    \item [(i)] $Ext^{C}(K)\subset Ext(K)$, where $Ext(K)$ denotes the exterior of $K.$
    \item [(ii)] $Ext^{C}(K)$ is open.
    \item [(iii)]  $Ext^{C}(K)= X\setminus cl^{C}(K).$
    \item [(iv)] $Ext^{C}(Ext^{C}(K))=int^{C}(cl^{C}(K)).$
    \item [(v)] If $K\subseteq H,$ then $Ext^{C}(H)\subseteq Ext^{C}(K).$
 \item [(vi)]$Ext^{C}(X\setminus Ext^{C}(K))=Ext^{C}(K).$
 \item [(vii)] $int^{C}(K)\subset  Ext^{C}(Ext^{C}(K)).$
 \item [(viii)] $Ext^{C}(K\cup H)\subset Ext^{C}(K) \cup Ext^{C}(H).$
 \item [(ix)]  $Ext^{C}(K\cap H)\supset Ext^{C}(K) \cap Ext^{C}(H).$
 \item [(x)]$X= int^{C}(K)\cup Ext^{C}(K)\cup Fr^{C}(K).$
  \end{enumerate}
\end{theorem}
\begin{proof}
 The property $(i)$, follow from  Definition \ref{13p12} and Theorem \ref{13p14}. The property $(ii)$, follow from  Definition \ref{13p7}. The property $(iii)$, follow from  Theorem \ref{13p15} and Definition \ref{13p12}. To prove $(iv)$ $Ext^{C}(Ext^{C}(K))=Ext^{C}(int^{C}(X\setminus K))= Ext^{C}(X\setminus cl^{C}(K) )=int^{C}(X\setminus (X\setminus cl^{C}(K) ))=int^{C}(cl^{C}(K)).$ To prove $(v)$ since $K\subseteq H,$ then $X\setminus H \subseteq X \setminus K,$ then $int^{C} (X\setminus H )\subseteq int^{C}(X \setminus K),$ hence,  $Ext^{C}(H)\subseteq Ext^{C}(K).$ For $(vi),$ $Ext^{C}(X\setminus Ext^{C}(K))= Ext^{C}(X\setminus int^{C}(X\setminus K)) =  int^{C}(X \setminus (X\setminus int^{C}(X\setminus K)))=int^{C}(int^{C}(X\setminus K))=  int^{C}(X\setminus K) = Ext^{C}(K).$  For $(vii),$ $int^{C}(K)\subseteq  int^{C}(cl^{C}(K))\subseteq int^{C}(X\setminus int^{C} (X\setminus K))=int^{C}(X\setminus Ext^{C}(K))=Ext^{C}(Ext^{C}(K)).$ The property $(viii)$  $Ext^{C}(K\cup H)=  int^{C}(X\setminus(K\cup H))= (X\setminus cl^{C}(K\cup H))= (X\setminus (cl^{C}(K)\cup cl^{C}(H)))\subset X\setminus cl^{C}(K)\cup X\setminus cl^{C}(H) =( X\setminus cl^{C}(K))\cap( X\setminus cl^{C}(H)))=Ext^{C}(K) \cap Ext^{C}(H)\subset Ext^{C}(K) \cup Ext^{C}(H).$ The property $(ix)$: $Ext^{C}(K\cap H)=  int^{C}(X\setminus(K\cap H))= (X\setminus cl^{C}(K\cap H))\supset (X\setminus (cl^{C}(K)\cap cl^{C}(H))) =( X\setminus cl^{C}(K))\cup( X\setminus cl^{C}(H)))=Ext^{C}(K) \cup Ext^{C}(H)\supset Ext^{C}(K) \cap Ext^{C}(H).$ The property $(x)$ follow from the Definitions \ref{13p11} and \ref{13p12}.

\end{proof}
In general $Ext(K)\nsubseteq Ext^{C}(K),$ $Ext^{C}(K\cup H)\nsupseteq Ext^{C}(K) \cup Ext^{C}(H)$ and $Ext^{C}(K\cap H)\nsubseteq Ext^{C}(K) \cap Ext^{C}(H).$ For example:
\begin{example}
By Example \ref{14p1}, there exists $K=\mathbb{I}$, such that $Ext(K)=int(\mathbb{R}\setminus \mathbb{I})=int(\mathbb{Q})=\mathbb{Q}$. But $Ext^{C}(K)=int^{C}(\mathbb{R}\setminus \mathbb{I})=int^{C}(\mathbb{Q})\subsetneqq \mathbb{Q}$, because $\mathbb{Q}$ is not $C$-open set. Hence, $Ext(K)\nsubseteq Ext^{C}(K).$ For $Ext^{C}(K \cup H)\nsupseteq Ext^{C}(K) \cup Ext^{C}(H)$, let $K=(-\infty,2)$ and $H=(0,\infty),$ then $Ext^{C}(K\cup H)=Ext^{C}(\mathbb{R})=\emptyset$ and $Ext^{C}(K) \cup Ext^{C}(H)=(2,\infty) \cup (-\infty,0).$ Hence, $Ext^{C}(K\cup H)\nsupseteq Ext^{C}(K) \cup Ext^{C}(H)$. For $Ext^{C}(K\cap H)\nsubseteq Ext^{C}(K) \cap Ext^{C}(H)$, let $K=(-\infty,2]$ and $H=[2,\infty),$ then $Ext^{C}(K\cap H)=Ext^{C}(\{2\})=\mathbb{R}\setminus \{2\}$ and $Ext^{C}(K) \cap Ext^{C}(H)=(2,\infty) \cap (-\infty,2)=\emptyset$, hence, $Ext^{C}(K\cap H)\nsubseteq Ext^{C}(K) \cap Ext^{C}(H).$
\end{example}

\section{$C$-continuity and $C$-compactness }
In this section, we define $C$-continuous functions and $C$-compact sets via the concepts of $C$-open and $C$-closed spaces and investigate their master properties.
\begin{definition}

A function $h:(X, \mathcal{T}) \to (Y,\mathcal{P})$ is said to be {\it $C$-continuous} if $h^{-1}(U)$ is $C$-open in $X$ for all open sets $U$ in $Y.$
\end{definition}

\begin{definition}
Let $(X, \mathcal{T})$ and $(Y,\mathcal{P})$ be topological spaces and $h:(X, \mathcal{T}) \to (Y,\mathcal{P})$ be a $C$-continuous functions. $h$ is {\it $C$-open} if and only if for all open sets $U\subseteq X,$ $h(U)$ is $C$-open set in $Y.$ $h$ is {\it $C$-closed} if and only if for all closed sets $V\subseteq X,$ $h(V)$ is $C$-closed set in $Y.$

\end{definition}

\begin{definition}

A bijection function $h:(X, \mathcal{T}) \to (Y,\mathcal{P})$ is said to be {\it $C$-hmoeomrphism} if and only if  $h$ and $h^{-1}$ are $C$-continuous.
\end{definition}
\begin{theorem}
 Any $C$-continuous function is continuous.
\end{theorem}
\begin{proof}
 Let $h:(X, \mathcal{T}) \to (Y,\mathcal{P})$ be a $C$-continuous function. Let $U$ be any open set in $Y$, then by $C$-continuity $h^{-1}(U)\in CO(X)$. Since any $C$-open set is open set, then $h^{-1}(U)\in \mathcal{T}$.
\end{proof}
\begin{theorem}

 Any $F$-continuous function is $C$-continuous.
\end{theorem}
\begin{proof}
  Obvious, because any $F$-open set is $C$-open.
\end{proof}

  It is clear, for any topological space $(X,\mathcal{T})$ we have have.
  \begin{center}
  $F$-continuous function $\Longrightarrow$ $C$-continuous function $\Longrightarrow$ continuous function.
\end{center}

 None of the above implications is reversible.\\

 Here an example of continuous function is neither $C$-continuous nor $F$-continuous.

 \begin{example}
 
  Let $(\mathbb{R},\mathcal{T_{\mathbb{I}}})$ be the excluded set topological space on $\mathbb{R}$ by $\mathbb{I}.$ Then the identity function $id:(\mathbb{R},\mathcal{T_{\mathbb{I}}})\to (\mathbb{R},\mathcal{T_{\mathbb{I}}})$ is continuous function, which is neither $C$-continuous nor $F$-continuous, because $\mathbb{Q}\in \mathcal{T_{\mathbb{I}}}$ is open and $id^{-1}(\mathbb{Q})=\mathbb{Q}$ is neither $C$-open nor $F$-open set, because $cl(\mathbb{Q})\setminus \mathbb{Q}=\mathbb{R}\setminus \mathbb{Q}=\mathbb{I}$ is uncountable set.
 \end{example}
 Also, there is an example of $C$-continuous function which is not $F$-continuous function.
 \begin{example}
 Let $(\mathbb{R},\mathcal{C})$ be the co-countable topological space (see \cite{Steen}). Let $id$ be the identity function such that $id:(\mathbb{R},\mathcal{C}) \to (\mathbb{R},\mathcal{C})$ and let $K$ be any open set in $(\mathbb{R},\mathcal{C})$, then $K$ is $C$-open, because $cl(K)\setminus K=\mathbb{R}\setminus K$ is countable. Hence, $id:(\mathbb{R},\mathcal{C}) \to (\mathbb{R},\mathcal{C})$ is $C$-continuous function. However, $\mathbb{R}\setminus\mathbb{Z}$ is open set such that $id^{-1}(\mathbb{R}\setminus\mathbb{Z})=\mathbb{R}\setminus\mathbb{Z}$ is not $F$-open set, because $cl(\mathbb{R}\setminus\mathbb{Z})\setminus (\mathbb{R}\setminus\mathbb{Z})=\mathbb{R}\setminus (\mathbb{R}\setminus\mathbb{Z})=\mathbb{Z}$ is not finite set.
 \end{example}
\begin{definition}

Let $X$ be a topological space, then $X$ is $C$-compact (resp. $C$-Lindel$\ddot{o}$f) if and only if any open cover of $X$ has a finite (resp. countable) subcover of $C$-open sets.
\end{definition}
\begin{definition}

Let $X$ be a topological space, then $X$ is $C$-countably compact if and only if any countable open cover of $X$ has a finite subcover of $C$-open sets.
\end{definition}
\begin{theorem}
  Any $C$-compact space is compact space.
\end{theorem}
\begin{proof}
  Obvious, because any $C$-open set is open.
\end{proof}
\begin{corollary}
    Any $C$-Lindel$\ddot{o}$f (resp., $C$-countably compact) space is Lindel$\ddot{o}$f (resp., countably compact) space.
\end{corollary}
\begin{theorem}\label{14p3}
  Any $F$-compact space is $C$-compact space.
\end{theorem}
\begin{proof}
  Obvious, because any $F$-open set is $C$-open.
\end{proof}
\begin{corollary}\label{13ff1}

    Any $F$-Lindel$\ddot{o}$f (resp., $F$-countably compact) space is $C$-Lindel$\ddot{o}$f (resp., $C$-countably compact) space.
\end{corollary}
\begin{theorem}\label{13p6}

  Let $h:(X, \mathcal{T}) \to (Y,\mathcal{P})$ is $C$-continuous, onto, $C$-open function and $(X, \mathcal{T})$ is $C$-compact, then $Y$ is $C$-compact.

\end{theorem}
\begin{proof}
Let $\{V_{\alpha} :\alpha \in \Lambda \}$ be any open cover of $Y$. Since $h$ is $C$-continuous,
then $h^{-1}(V_{\alpha})$ is $C$-open in $X$ for each $\alpha \in \Lambda$. Since $Y\subseteq \bigcup_{\alpha \in \Lambda} V_{\alpha} $, then $X=h^{-1}(Y)\subseteq h^{-1}(\bigcup_{\alpha \in \Lambda} V_{\alpha})=\bigcup_{\alpha \in \Lambda} h^{-1}(V_{\alpha})$ ), that is means $\{h^{-1}(V_{\alpha}): \alpha \in \Lambda \}$ is an open cover of $X$ (because any $C$-open set is open). Then, by the $C$-compactness of $X$, there exist $\alpha_{1},\alpha_{2}, \cdots \alpha_{n} \in \Lambda$ such that $h^{-1}(V_{\alpha_{1}})\cup h^{-1}(V_{\alpha_{2}})\cup \dots \cup h^{-1}(V_{\alpha_{n}})=X$, then $h[h^{-1}(V_{\alpha_{1}})\cup h^{-1}(V_{\alpha_{2}})\cup \dots \cup h^{-1}(V_{\alpha_{n}})]=h(X)$, then $h(h^{-1}(V_{\alpha_{1}}))\cup h( h^{-1}(V_{\alpha_{2}}))\cup \dots \cup h( h^{-1}(V_{\alpha_{n}}))=Y$, then $V_{\alpha_{1}}\cup V_{\alpha_{2}}\cup \dots \cup V_{\alpha_{n}}=Y$. Since $h$ is $C$-open, then $\{V_{\alpha_{1}}\cup V_{\alpha_{2}}\cup \dots \cup V_{\alpha_{n}}\}$ is a finite subcover of $C$-open sets for $Y$. Therefore, $(Y,\mathcal{P})$ is a $C$-compact space.
\end{proof}
A subset $B$ of a space $X$ is $C$-compact if and only if $B$ is a $C$-compact topological space with the subspace topology.
\begin{theorem}
  
  Let  $h:(X, \mathcal{T}) \to (Y,\mathcal{P})$ is $C$-continuous function and $(X, \mathcal{T})$ is $C$-compact, then $Y$ is compact.
\end{theorem}
\begin{proof}
   Using the same proof of Theorem \ref{13p6}. 
\end{proof}
\begin{theorem}
 Let  $h:(X, \mathcal{T}) \to (Y,\mathcal{P})$ is $F$-continuous, $F$-open, onto function and $(X, \mathcal{T})$ is $F$-compact, then $Y$ is $C$-compact.

\end{theorem}
\begin{proof}
 The proof very clear, by Theorem \ref{14p3} any $F$-compact space is $C$-compact, and Theorem $13$ in \cite{Mesfer2} any $F$-continuious, $F$-open and onto image of $F$-compact space is $F$-compact space.
\end{proof}
\begin{theorem}
 Let $h:(X, \mathcal{T}) \to (Y,\mathcal{P})$ is $C$-continuous, $C$-open, onto function and $(X, \mathcal{T})$ is $C$-Lindel$\ddot{o}$f (resp., $C$-countably compact) space, then $Y$ is  $C$-Lindel$\ddot{o}$f (resp., $C$-countably compact).

\end{theorem}
\begin{proof}
  Using the same proof of Theorem \ref{13p6}.
\end{proof}
\begin{theorem}
Let $h:(X, \mathcal{T}) \to (Y,\mathcal{P})$ is $F$-continuous, onto, $F$-open function, and $(X, \mathcal{T})$ is $F$-Lindel$\ddot{o}$f (resp., $F$-countably compact) space, then $Y$ is $C$-Lindel$\ddot{o}$f (resp., $C$-countably compact).
\end{theorem}
\begin{proof}

 Obviously, by Corollary \ref{13ff1} any $F$-Lindel$\ddot{o}$f (resp., $F$-countably compact) space is $C$-Lindel$\ddot{o}$f (resp., $C$-countably compact) space, and Theorem $15$ in \cite{Mesfer2} any $F$-continuious, $F$-open and onto image of $F$-Lindel$\ddot{o}$f (resp., $F$-countably compact) space is $C$-Lindel$\ddot{o}$f (resp., $C$-countably compact) space.
\end{proof}
\begin{theorem}
Let $h:(X, \mathcal{T}) \to (Y,\mathcal{P})$ is $C$-continuous and onto function and $(X, \mathcal{T})$ is $C$-Lindel$\ddot{o}$f (resp., $C$-countably compact) space, then $Y$ is Lindel$\ddot{o}$f (resp., countably compact).

\end{theorem}
\begin{proof}
 Using the same proof of Theorem \ref{13p6}.
\end{proof}

 Obviously, for any topological space $(X,\mathcal{T})$ we have have.
  \begin{center}
 $F$-compactness $\Longrightarrow$ $C$-compactness $\Longrightarrow$ compactness.
\end{center}

 None of the above implications is reversible. Here an example of compact space is neither $C$-compact nor $F$-compact spaces.
\begin{example}
  Overlapping Interval Topology \cite{Steen}. On the set $X=[-1,1]$ we generate a topology from sets of the form $[-1,b)$ for $b>0$ and $(a,1]$ for $a<0.$ Then all sets of the form $(a,b)$ are also open. $X$ is compact, since in any open covering, the two sets which include $1$ and $-1$ will cover $X$. The space $X$ is neither $C$-compact nor $F$-compact spaces, because there exists $\{[-1,0.5),(-0.5,1]\}$ is an open cover for $X$ has no finite subcover of $C$-open sets or $F$-open sets, because $[-1,0.5)$ and $(-0.5,1]$ are neither $C$-open nor $F$-open sets ($cl[-1,0.5)\setminus [-1,0.5)=[-1,1]\setminus [-1,0.5)=[0.5,1]$ is neither countable nor finite set and $cl(-0.5,1]\setminus (-0.5,1]=[-1,1]\setminus (-0.5,1]=[-1, -0.5]$ is neither countable nor finite set).
\end{example}

\section{ Conclusion and other tasks}
In this article, we have displayed the notions of $C$-open sets and discussed its master
properties. Then, we have defined some operators via $C$-open and $C$-closed sets. We have disclosed the relationships between these operators and examined their mains features. In addition, we have introduced continuous functions and compact sets using $C$-open and $C$-closed sets and checked their main properties. \\

Our next works will concentrate on studying further topological concepts by
$C$-open and $C$-closed sets on separation axioms, connectedness, and the other topological properties.

\end{document}